\documentclass[11pt]{scrartcl}

\usepackage{amsthm}
\newtheorem{thm}{Theorem}
\newtheorem{theorem}[thm]{Theorem}

\newtheorem{corollary}[thm]{Corollary}

\newtheorem{lemma}[thm]{Lemma}

\theoremstyle{definition}

\newtheorem{example}{Example}
\theoremstyle{remark}

\newtheorem{remark}{Remark}

\PassOptionsToPackage{override}{xcolor} 
\usepackage[utf8]{inputenc}
\usepackage{hyperref}
\usepackage{subfig}
\usepackage{amsmath} 
\usepackage{amssymb}
\usepackage{amsfonts,amssymb}
\usepackage{dsfont}
\usepackage[mathscr]{euscript}
\usepackage{enumerate,xspace,threeparttable}
\usepackage{graphicx}
\usepackage{verbatim}
\usepackage{algorithmic}
\usepackage{listings}
\usepackage{wrapfig}
\usepackage{translator}

\providecommand{\qr}{\eqref}
\renewcommand{\d}{\,d}

\providecommand{\RR}{\mathbb{R}}

\providecommand{\ZZ}{\mathbb{Z}}

\providecommand{\VV}{\mathbb{V}}

\providecommand{\EE}{\mathsf{E}}

\providecommand{\CF}{\mathscr{F}}

\providecommand{\mscr}{\mathscr}
\providecommand{\mc}{\mathcal}
\providecommand{\mb}{\mathbf}

\providecommand{\mbb}{\mathbb}

\providecommand{\opn}{\operatorname}

\def\ii#1{^{(#1)}}

\providecommand{\E}{\mathsf{E}}

\providecommand{\Ex}[1]{\E\left(#1\right)}

\providecommand{\var}{\opn{var}}

\providecommand{\msf}{\mathsf}
\providecommand{\ett}{\mathsf{1}}
\providecommand{\Tr}{\opn{Tr}}
\providecommand{\Ker}{\opn{Ker}}

\providecommand{\Ordo}[1]{{O(#1)}}

\providecommand{\tl}{\tilde}

\providecommand{\scp}[2]{\left\langle#1,#2\right\rangle}

\providecommand{\T}{^{\!\mathrm{T}}}

\renewcommand{\L}{\mathscr{L}}

\def\EE{\mbb{E}}

\def\B{\mscr{B}}
\def\A{\opn{\msf{A}}}
\def\a{A}

\def\Ex#1{\EE\left[#1\right]}

\def\X{\mscr X}
\def\T{\opn{\msf{T}}}
\def\M{\opn{M}}

\def\LL{\mbb L}

\def\E{{\opn{\mscr{E}}}}
\def\R{\mscr R}
\def\Q{\opn{Q}}

\def\L{\opn{\mscr L}}

\def\scp#1#2{\left\langle #1 , #2 \right\rangle}
\def\norm#1{\left\| {#1} \right\|}
\def\Scp#1#2{\left( #1 \mid #2 \right)}

\def\scpe#1#2{\left\langle {#1} , {#2} \right\rangle_{\E}}
\def\norme#1{\left\| {#1} \right\|_{\E}}

\def\al{\alpha}
\def\Al{\A(\al)}

\usepackage{marginnote}

\begin{document}
\title{Ergodic Theory of Kusuoka Measures}
\author{Anders Johansson\footnote{Department of Mathematics,
    University of G\"avle, 801 76 G\"avle, Sweden\newline email:
    ajj@hig.se}, Anders \"Oberg\footnote{Department of Mathematics,
    Uppsala University, P.O. Box 480, 751 06, Uppsala, Sweden\newline
    email: anders@math.uu.se} and Mark Pollicott\footnote{Mathematics
    Institute, University of Warwick, Coventry, CV4 7AL, UK\newline
    email: mpollic@maths.warwick.ac.uk}} \date{\today}
\maketitle

\begin{abstract}\noindent
{\bf Abstract.}
  In the analysis on self-similar fractal sets, the Kusuoka measure
  plays an important role (cf.  \cite{kusuoka2}, \cite{kajino},
  \cite{str3}). Here we investigate the Kusuoka measure from an
  ergodic theoretic viewpoint, seen as an invariant measure on a symbolic space.
  Our investigation shows that the Kusuoka
  measure generalizes Bernoulli measures and their properties to
  higher dimensions of an underlying finite dimensional vector space. 
  Our main result is that the transfer operator on functions has a 
  spectral gap when restricted to a certain Banach space that contains the
  H\"older continuous functions, as well as the highly discontinuous $g$-function
  associated to the Kusuoka measure. As a consequence, we obtain
  exponential decay of correlations. In addition, we provide some
  explicit rates of convergence for a family of generalized Sierpi\'nski
  gaskets.
\end{abstract}
\tableofcontents
\section{Introduction}

\subsection{Background and problems}
The Kusuoka measure has recently attracted some attention, since it
gives rise to a well-working Laplacian on fractal sets (see, e.g.,
\cite{str}). The Laplacian is usually defined weakly with respect to a
measure on the fractal set. A standard way of accomplishing this is to
first define a Dirichlet energy form $\E(f,f)$ on the fractal $K$, in
analogy with $\int |\nabla f|^2\d\mu(x)$, and then to define the
Laplacian by equating the corresponding bilinear form ${\mc E} (u,v)$
with $-\int \left(\Delta_{\mu}u\right)v d\mu$, for functions $v$
vanishing on the boundary.  It is well-known that with respect to the
normalized Hausdorff measure, the domain of the Laplacian is not even
closed under multiplication. By contrast, the Kusuoka measure is
well-behaved in this sense and in some other more subtler ways, e.g.,
for the Laplacian it provides Gaussian heat kernel estimates with
respect to the effective resistance metric and can be regarded as a
second order differential operator \cite{kigami2}.

Recently, Strichartz and his collaborators (\cite{str3}, \cite{str2})
have proved some basic properties of the Kusuoka measure. Here we
provide an investigation of the Kusuoka measure from the point of view
of ergodic theory on symbolic shift spaces. 
For instance, we provide exponential mixing results
as a consequence of the quasi-compactness of a transfer operator as it
acts on a Banach space which contains functions that may have a dense
set of discontinuity points, but which can be regarded as ``smooth''
when they are integrated with respect to the Kusuoka measure. In fact, the associated 
  transfer operator is given by a simple multiplication when acting on a 
  certain space of matrix-valued processes.   However,  when restricting the transfer 
  operator to ordinary functions, 
  the corresponding transition probability
  function has a dense set of discontinuity points, which presents
  difficulties.

Our abstract way of treating the Kusuoka measure is rather similar to
the one in the original work by Kusuoka (\cite{kusuoka2}) and covers
in fact a general class of measures that can be defined by
products of matrices. We point out that the Kusuoka measure is really
a family of measures that generalizes the Bernoulli measures to higher
dimensions. We also note that the theory of matrix product state
representations of quantum Potts models (see e.g.\ \cite{mps}) seems
to be quite related, although we have not used any particular result
from this theory.

We believe that our analysis opens the door to interesting further research.  For example, it should now be possible to compute the
entropy of the measure explicitly. In view of our exponential
mixing results, it should then be possible to provide a multifractal
formalism for the Kusuoka measure. 
A major challenge would be to
generalise the type of results we provide here for matrices (as our
restriction maps) to infinite dimensional operators.
Using infinite dimensional operators, one could hope to be able treat
the Kusuoka measure on fractal sets with infinite boundaries, such as that of the
Sierpi\'nski carpet. However, it is not immediately clear how one should
define the Kusuoka measure even in the case of the Vicsek set, which
has a countably infinite boundary.

Other challenges in the fractal realm would include, e.g., the
problem of relating our results to the Cartesian product of a
Sierpi\'nski gasket with itself, or if one glues together the boundary
points of two such copies, producing a ``fractafold''.

\subsection{Summary of the main results}
We prove quasicompactness of a transfer operator defined on a Banach space, 
with a norm that is 
an integrated H\"older norm in terms the variations of functions on cylinder sets 
of a symbolic space.
In some sense, we are studying the transfer operator of a space of ``Besov'' type, 
since the moral is that we look at a ``smooth'' space that may have many 
discontinuities (since we integrate), and this is necessary in order to handle
the dense set of discontinuity points of the $g$-function that defines the transfer operator. 

To be more precise: Let $S$ be a finite set and let 
$\X$ denote the symbolic space $\X=S^{\ZZ_+}$ 
($\ZZ_+=\{0,1,2,\dots\}$) of functions $x:\ZZ_+\to S$.
The (point)
shift map $\T:\X\to\X$ is defined as $(\T x)(n) = x(n+1)$. 
In our abstract setting, the Kusuoka measure $\nu$ (\cite{kusuoka2}) is a shift-invariant
measure on the space $\X$. 
The transfer operator $L$ is the dual of the shift operator $\T f = f\circ T$ 
on the Hilbert-space of functions $L^2(\X,\nu)$.  It has the form
$$ Lf(x) = \sum_{s\in S} g(sx) f(sx) $$
where the \emph{$g$-function} can be defined as 
$$
g(x) = \lim_{n\to\infty} \frac{\nu([x]_n)}{ \nu\left([\T x]_{n-1}\right)},
$$
where $[x]_n$ denote the cylinder of length $n$ containing $x$. 

Given a real number $\gamma$, $0 < \gamma < 1$, we define for
$f\in L^2(\X,\nu)$ 
a Banach space $L^2_\gamma \subset L^2$ with norm ${\norm f}_{L^2_\gamma}$ by
$$
  {\norm f}_{L^2_\gamma} = 
  \sum_{n=0}^\infty \gamma^{-n} \|{f\ii n}\|_{L^2}.
$$
Here $(f\ii n)_{n=0}^\infty$
is the martingale difference sequence for $f$ 
given by $f\ii 0 = f_0$ and $f\ii n(x) = f_n(x) - f_{n-1}(x)$ for
$n\geq 1$, and where $x\mapsto f_n(x):=f([x]_n)$ is the orthogonal projection of $f$ onto the
finite dimensional subspace $L^2_n$ of $L^2(\X,\nu)$ of $\CF_n$-measurable functions, where $\CF_n$ is the $\sigma$-algebra generated by cylinder sets with length $n$.

Quasicompactness of $L$ on our ``Besov space'' $L_\gamma^2$, means that there exists $0 < \rho < 1$ such that for any $f\in L_\gamma^2$, where $\gamma$ is sufficiently close to one, we have
\begin{equation}\label{qc}
  \left\| L^m f - \int f d\mu \right\|_{L_\gamma^2} \leq C \rho^m
  \left\|f\right\|_{L_\gamma^2}
\end{equation}
for a uniform constant $C$. 

We prove \qr{qc} by representing $L$ as a \emph{dilation} of a 
transfer operator $\L$ defined on a larger graded Hilbert-space 
$\VV=\overline{\lim\VV_n}$ 
consisting of matrix-valued processes.
The graded Hilbert space $L^2 = \overline{\lim L^2_n}$ 
is isometrically embedded into $\VV$.
It is fairly straightforward to show that 
quasi-compactness holds for $\L$ on $\VV_\gamma$ and, 
since $L = \Q \circ \L$ 
where $\Q:\VV\to L^2$  is the orthogonal projection.
This result carries over to $L$ on $L^2_\gamma$ for those $\gamma$ such that $\Q$
is continuous as an operator from $\VV_\gamma$ to $L^2_\gamma$. 

From the quasicompactness result \qr{qc}, exponential 
decay of correlations (mixing at an exponential rate) 
follows automatically: If $f\in L^2(\X,\nu)$ and 
$g \in L_\gamma^2$ then for some $0 < \rho < 1$ and
some uniform constant $C$, we have
$$\left|\int f \, (g\circ T^n) \; d\nu- \int f\; d\nu \int g\; d\nu\right|\leq C\rho^n.$$

We note that our quasicompactness results depend on the general symbolic formulation,
where we use the ultrametric on the symbolic space $\X$ and not some underlying geometric distance. 
Hence, the quasicompactness on our ``Besov space'' will not immediately translate into quasicompactness on a Besov space defined on the metric space of an underlying fractal, such as those of Jonsson \cite{jonsson} and Grigor'yan \cite{grigor}.

\subsection{More results and the structure of the paper}
In section 2, we present the Kusuoka measure from an abstract point of view, namely on cylinder sets, which corresponds to the products of matrices that act on a finite-dimensional vector space that corresponds to a space of harmonic functions. We give two special examples. The first shows that the Kusuoka measure in one dimension reduces to the class of Bernoulli measures. We can thus view abstract the Kusuoka measure as a natural generalisation of the Bernoulli measure, the difference being that we ``multiply matrices instead of numbers''. The second example is a brief discussion of a well-studied case, that of the Sierpi\'nski gasket, extensively studied in \cite{str1}, \cite{str3}, \cite{str}, \cite{str2}, and in many other works.

In section 3, we state the main results: quasicompactness of the transfer operator on the space $L_\gamma^2$, as well as exponential decay of correlations. We also consider special results for the Sierpi\'nski gasket and the family $SG_n$, defined in subsection 3.2. In Theorem \ref{cor4}, we obtain precise mixing rates of convergence in a simplified case, when we shift cylinders of a fixed length. We have only stated this result for the Sierpi\'nski gasket, but we have made some calculations for the mixing rates for $SG_n$, $n=3,4,5$; see Example~\ref{SG}.

In section 4, we introduce a Hilbert space $\VV$ on which a transfer operator that acts on matrix-valued operators is easily analysed in terms of the matrix operator $\M$, defined in \eqref{Mdef}. Here we obtain a simple expression for the transfer operator as the dual of the shift map $\T$, so here the ``higher-dimension'' generalisation of Bernoulli measures is exploited. The proof of Theorem~\ref{contr1} that states the quasicompactness on a space $\VV_\gamma$, equipped with a certain ``smooth'' norm $\|\cdot \|_\gamma$, relies essentially the contraction of the matrix-operator $\M$ and the contraction ratio $\theta_1<1$ remains the same. In subsection 4.5, we obtain a strict contraction of $\M$ acting on symmetric matrices and this is used to obtain more precise rates of convergence (Theorem~\ref{cor4}) in the case of the Sierpi\'nski gasket.

In section 5, we prove that the quasicompactness result in section 4, for the matrix-valued space $\VV_\gamma$, may be retrieved for {\em functions} in $L^2_\gamma$, by means of a projection; see Lemma~\ref{thmD} and its proof subsection 5.1, the most technical and difficult part of the paper. In Lemma~\ref{thmD} a new contraction factor $\theta_2<1$ is introduced and the final contraction ratio $\rho$ expressed in terms of the quasicompactness of Theorem~\ref{thm1} must be strictly larger than both $\theta_1$ and $\theta_2$. It remains an open problem, even in the case of the Sierpi\'nski gasket, whether $\theta_2=\theta_1$. In subsection 5.2, we restrict our attention to the Sierpi\'nski gasket and prove Theorem~\ref{cor4}.

\subsection{Acknowledgements} 

The problem of studying Kusuoka measures using transfer operator techniques was proposed by Professor R.S.\ Strichartz of Cornell University. The second author is grateful for visits to Cornell in May 2012 (when the problem was proposed) and again in September 2014 and December 2015. This research was supported by the {\bf Royal Society (UK)}, grant {\bf IE121546}: {\em Ergodic theory of energy measures on fractals}. The grant provided several opportunities for all authors to visit University of Warwick, Uppsala University and University of G\"avle. We are grateful to the anonymous referee for many clarifying comments.

\section{The Kusuoka measure} \label{kusdef}
\subsection{Cylinders and cylinder sets}
\label{sec:notation}
 An elementary \emph{cylinder} is a function $\al: [a,b)\to S$ defined
on some integer interval $[a,b)=\{a,a+1,\dots,b-1\}$. The \emph{length} of
the cylinder is $\ell(\al)=b-a$. The corresponding \emph{cylinder set}
$\al\subset \X$ is the set of $x\in\X$ that coincides with $\al$ on
$[a,b)$. (Notice that we make no notational distinction between a
cylinder and the equivalent cylinder set.)

A cylinder is an \emph{initial} cylinder if the domain is
$[a,b)=[0,n)$ for some $n$ and we write $S^n$ instead $S^{[0,n)}$, and
also $S^*$ for the set $\cup_n S^n$ of initial cylinders. The set
$S^0$ consists of the empty cylinder $\emptyset$.  Let $[x]_n$ denote
the initial cylinder obtained by restricting $x$ to the interval
$[0,n)$. Let $\CF_n$ be the algebra generated by the cylinder sets
$[x]_n$, $x\in\X$ and let $\CF$ be the limit $\sigma$-algebra as $n\to \infty$.

For a cylinder $\al\in S^{[a,b)}$ and a symbol $s\in S$, an expression
of the form $\al s$ it is understood as the concatenation of the
cylinder with the symbol to the right, so that $\al s$ is a cylinder
in $S^{[a,b+1)}$ with $(\al s)(b)=s$. If $a>0$ then $s\alpha$ is the
corresponding concatenation to the left, but, if $\al\in S^n$ is an
initial cylinder then $s\al\in S^{n+1}$ with $(s\al)(0)=s$ and
$(s\al)(k)=\al(k-1)$, $k=1,\dots,n+1$. The expression $sx$ refers in
the same way to the concatenated and shifted sequence $sx\in\X$, where
$(sx)(0)=s$ and $sx(n) = x(n-1)$, $n\geq1$.

\subsection{Construction of an abstract Kusuoka measure}
In order to define the Kusuoka measure, we consider a fixed finite
dimensional Hilbert space $H$ having scalar product
$\scp\cdot\cdot$. Let $\B=\B(H)$ denote the space of bounded operators
on $H$.  For any cylinder $\al\in S^{[a,b)}$, we associate the
compound ``restriction map''
$$\Al=\a_{\al({a})}\, \cdots \a_{\al(b-1)},$$
where $\a_s\in \B$, $s\in S$ are operators with certain properties specified later. 
We define the \emph{Kusuoka measure} $\nu$ on the cylinder set
$\al=\{ x : x\vert_{[a,b)}=\al\} \subset \X$ as the trace
\begin{equation}
  \nu(\al) = \Tr\left(\,\Al^*\,\E \, \Al\,\right),\label{nudef}
\end{equation}
where $\E$ is a positive definite symmetric operator $H\to H$ such
that $\Tr(\E)=1$.

The definition \qr{nudef} defines a consistent probability measure on
the measurable space $(\X,\CF)$ if and only if the system
$\{\a_s: s\in S\}$ of maps satisfies the following two conditions
\begin{equation}\label{fundeq1}
  \sum_s \a_s^*\E A_s = \E
\end{equation}
and
\begin{equation}\label{fundeq2}
  \sum_s \a_s\a_s^* = I.
\end{equation}
Consistency of definition of $\nu$ follows: E.g.\ \qr{fundeq2} gives
that
$$\sum_s \nu(s\alpha) = 
\sum_s \Tr(\a_s^* \Al^*\E \Al \a_s) = \Tr(\Al^*\E\Al I)=\nu(\al), $$
so $\nu$ is consistent with extensions to the left. Similarly,
\qr{fundeq1} shows that $\sum_s\nu(\al s) = \nu(\al)$.  It is also
clear that $\nu$ will be a shift invariant measure on $(\X,\CF)$, since
$\nu(\al)$ is determined by the \emph{word} corresponding to the
cylinder $\al\in S^{[a,b)}$.

As is shown in \cite{kusuoka2}, the Kusuoka measure is moreover
\emph{ergodic} if one assumes that the system is \emph{irreducible} in
the sense that
\begin{equation}\label{irredorig}
  \text{the linear maps $\a_s$, $s\in S$, have no common nontrivial invariant subspace $W$.}
\end{equation}
That is, there exists no subspace $W$, $(0)\subsetneq W \subsetneq H$,
such that $\a_s(W)\subset W$ for all $s\in S$.

We will consider the space $\B=\B(H)$ of operators on $H$. Note that,
if we define the operators $\M: \B\to\B$ and $\M^*:\B\to\B$ by 
\begin{equation}\label{Mdef}
  \M(B) = \sum_s \a_s B \a_s^*,\quad\text{and}\quad
  \M^*(B) = \sum_s \a_s^* B \a_s
\end{equation}
then \qr{fundeq1} and \qr{fundeq2} can be expressed as a statement of
fixed points, i.e.\ that $\M(I)=I$ and $\M^*(\E)=\E$. The operator
$\M^*(B)$ is the adjoint of $\M$ on $\B$ with respect to the
Hilbert-Schmidt scalar product $\scp AB_{HS}=\Tr(B^*A)$.

We will often use the the scalar product $\scpe \cdot\cdot$ with
associated norm $\norme A = \scpe AA^{1/2}$ given by
\begin{equation}\label{scpedef}
  \scpe AB = \Tr(\E AB^*) =\Tr(B^* \E A),\quad A,B\in\B. 
\end{equation}
Notice that $\nu(\al)=\norme{\Al}^2$ and that \qr{fundeq1} and
\qr{fundeq2} are equivalent to the statement that the scalar product
$\scpe\cdot\cdot$ is ``bi-invariant'' in the sense that
\begin{equation}\label{scpfundeq1}
  \scpe XY 
  = \sum_{\al\in S^k} \scpe{\Al X}{\Al Y} 
  = \sum_{\al\in S^k} \scpe{X\Al }{Y\Al} 
  , \quad\forall\,
  X,Y\in \B,
\end{equation}
for all $k\geq 0$.

For our main results, we use an irreducibility condition, implying \eqref{irredorig},
stating that for some $k>1$
\begin{equation}\label{irred}
  c_k=\inf_{F} \sum_{\al\in S^k} \scpe{\Al F}{\Al}^2 > 0
\end{equation}
where the infimum (minimum) is over the compact set of all symmetric
operators $F\in\B$ such that $\norme F=1$ and $\scpe FI=0$. Notice
that $c_k < 1$, since, by Cauchy--Schwarz, we have
$$
\sum_{\al\in S^k} \scpe{\Al F}{\Al}^2 \leq \sum_{\al} \norme {\Al F}^2
\cdot \norme{\Al}^2 = \sum_{\al} \norme {\Al F}^2 \cdot \nu(\alpha),
$$
where we conclude from the irreducibility condition \eqref{irredorig} that $\nu(\alpha)<1$. Moreover, \qr{scpfundeq1} implies that
$\sum_{\al\in S^k} \norme{\Al F}^2$ equals $\norme F^2 = 1$.

It is not clear to us in what circumstances the condition \qr{irred}
is a consequence of the irreducibility condition \qr{irredorig}.
Note that the stronger irreducibility condition \eqref{irred} follows if 
the maps $B\mapsto A_s B A_s^*$ have no non-trivial common invariant subspace of $\mscr B$.

\subsection{Examples}

The Kusuoka measure can usefully be viewed as a general construction
for a large class of shift invariant measures.
\begin{example}[Bernoulli measure]\quad The product form of Kusuoka
  measure shows that it is a natural generalisation of the Bernoulli
  measure.  Indeed, in the special case when $H=\RR$ and $A_s$ is
  $v\mapsto q_s v$, where \qr{fundeq2} states that
  $q_1^2 + \cdots + q_k^2=1$. In this case, $\nu$ is the Bernoulli
  measure associated to the distribution $p(s)=q_s^2$ on $S$. The energy operator is here the identity operator, which clearly has trace 1.
  Notice that the irreducibility condition \eqref{irred} is trivially satisfied in this case.
\end{example}

\begin{example}[Classical Kusuoka measure on $SG$]
The terminology we use comes from applications in the context of
  harmonic analysis on certain fractals: The space $H$ is the finite
  dimensional space of harmonic functions modulo constants on a
  self-similar fractal $K$ with a prescribed finite ``boundary''.
  
  The \emph{restriction map} $\A_s$ for a symbol $s\in S$ represents
  the restriction of harmonic functions to one of the $|S|$
  sub-fractals $K_s$, $s\in S$. The quadratic form
  $\E$ on $H$ is an \emph{energy form} which the harmonic functions in
  $H$ are minimising. By self-similarity we have an isomorphism $K_s\cong K$ and, by using this isomorphism and a suitable scaling, we
  can represent the restriction of harmonic functions to $K_s$ as a map $\A_s: H \to H$.  The invariance relation \qr{fundeq1} follows since the
  energy on the whole fractal is the sum of the energies on the
  sub-fractals. There is also a unique dual invariant form $\mscr R$ on $H^*$, but we identify $\mscr R$ with a given inner product on the Hilbert space $H$. Hence we obtain the relation \eqref{fundeq2}. 

A well-studied example is the Sierpi\'nski gasket, $SG$, which is the unique nonempty compact set satisfying
$$SG=\bigcup_{i=0}^2 F_i SG,$$
where $F_i =\frac{1}{2}(x+q_i)$, and where $\{q_i\}_{i=0}^2$ are the
vertices of an equilateral triangle. These three points are also the
boundary points of $SG$.  We obtain the Kusuoka measure on $SG$ (see,
e.g., \cite{str1}, \cite{str3}, \cite{str}, \cite{str2}) in the
special case $S =\{0, 1, 2\}$ and corresponding matrices
$\a_s = R^{-s} DR^s$, where $R$ is the rotation by $2\pi/3$, and where
$$
D= \left(
  \begin{matrix}
    \frac{3}{\sqrt{15}} & 0 \\
    0 & \frac{1}{\sqrt{15}}
  \end{matrix}
\right).
$$
For a non-zero harmonic function $h$, the
{\em energy measure} $\nu_h$ is defined on an elementary cylinder
$[w]=\{x:[x]_k=w\}$ by
$$ \nu_h([w]) = \E ( A_w h, A_w h),$$
where we have lifted the restriction of a harmonic function on
$F_w SG=F_{w_1}\cdots F_{w_n}SG$ to an element $A_w h$, which is also
a harmonic function on $SG$. We have
$$ \E(h,h) = \sum_{w\in S^k} \E(A_w h, A_w h), $$ 
where one should observe that the usual normalisation constants are built into the restriction maps $A_s$, and also that
$$ A_w = A_{w_k} A_{w_{k-1}} \dots A_{w_1}, $$
if $w=w_1w_2w_3\dots w_k$. We can for instance choose the basis of two
harmonic functions (see \cite{str2})
$h_1=\frac{\sqrt{2}}{3}(1,-\frac{1}{2},-\frac{1}{2})$ and
$h_2=\frac{1}{\sqrt{6}}(0,1,-1)$. We can obtain the Kusuoka measure on
$SG$ as the sum $\nu=\nu_{h_1}+\nu_{h_2}$ of energy measures for the
orthonormal basis of $H$.
In this case the restriction maps are symmetric matrices, whence
$\M=\M^*$ and $\E=(1/2)I$. We obtain by direct computation that the
action of $\M$ on the subspace of symmetric matrices in $\B$ is given
by
\begin{equation}\label{SGM}
\M\left(\begin{bmatrix} a &  b \\ b & -a \end{bmatrix} +
c I \right) = 
\frac45 \cdot \begin{bmatrix} a &  b \\ b & -a \end{bmatrix} 
+ cI. 
\end{equation}
Similarly, it contracts with a factor $4/5$ on the space of
anti-symmetric matrices. It follows that $\M$ acts as a contraction on
the space of trace-less matrices with the contraction constant
$\theta_1=4/5$, which is one of the constants that will be important to us in the sequel in order to describe mixing rates. From this it follows (Corollary 5 below) that
if $A\in \CF_k$ (measurable with respect to cylinders sets of length $k$) and $B\in \CF$ (any Borel set), we have
$$\left|\nu(\T^{-(n+k)}A \cap B)-\nu(A)\nu(B)\right|\leq 2\left(\frac{4}{5}\right)^n.$$
\end{example}

\section{Results}\label{sec:results}

\subsection{A spectral gap for the transfer operator on the
  associated Banach space}
A standard approach to studying ergodic properties of $\T$-invariant
measures on $\X$ is to use transfer operators defined on spaces of
functions.  In particular, if we consider the real Hilbert space
$L^2(\X, \nu)$ with the scalar product $\scp fg = \int f g \d\nu$ and
norm $\| f\|={\scp ff}^{1/2}$ then we can define the \emph{transfer
  operator} $L : L^2(\X, \nu) \to L^2(\X, \nu) $ as the dual of the
shift map, i.e.,
\begin{equation}
  \scp{L f}g =\scp{ f }{ g \circ \T}\label{dualityL}
\end{equation}
for $f,g \in L^2(\X, \nu)$.  It is easy to see that that the operator
norm of $L$ is one and that it has a maximum modulus eigenvalue with
the constant function $1$ as the normalised eigenvector.

The operator $L$ takes the explicit form
$$ Lf(x) = \sum_s g(sx) f(sx) $$
where the \emph{$g$-function}, $g:X\to [0,1]$, can be defined as
$$g(x)  = \lim_{n\to\infty} \frac{\nu([x]_n)}{ \nu\left([\T x]_{n-1}\right)}. $$
The $g$-function exists, on account of the martingale convergence
theorem, $\nu$-almost everywhere.  Bell, Ho and Strichartz \cite{str3}
showed that the $g$-function associated to the Kusuoka measure for the
Sierpi\'nski gasket has a dense countable family of discontinuities. In
particular, the Kusuoka measure is not a Gibbs measure and therefore
not amenable to the classical thermodynamic ideas.

We say that an operator $L$ on a Banach space has a \emph{spectral
  gap} if it has a unique eigenvalue $\lambda$ of maximum modulus and
if all other elements of the spectrum of $L$ has modulus less than
some $\rho<|\lambda|$.  In order to prove that there is a spectral gap
for the operator $L$, one usually needs to restrict it to a smoother
class of functions which is considerably smaller than $L^2$.  For the
Kusuoka measure, because of the discontinuities in the $g$-function,
it is not appropriate to consider, say, H\"older continuous
functions. Instead we consider functions where the martingale sequence
converges in $L^2$-norm quickly enough.

Any element $f\in L^2(\X,\nu)$ can be uniquely represented by the
corresponding \emph{martingale process}
$$ f(\al)=\Ex{f\mid \al} = \nu(\al)^{-1} \int_\al f\d\nu,\quad\al\in S^*.$$
We will usually refer to the martingale process $f(\al)$ by the same
name as the element $f(x)$ in $L^2(\X,\nu)$.  The function
$x\mapsto f_n(x):=f([x]_n)$ is the orthogonal projection onto the
finite dimensional subspace $\mathrm m \CF_n$ of $L^2(\X,\nu)$ of
$\CF_n$-measurable functions and by the martingale convergence
theorem, we have $\lim_n f_n(x) = f(x)$, $\nu$-almost everywhere.  The
\emph{martingale difference sequence} $(f\ii n)_{n=0}^\infty$ of $f$
is given by $f\ii 0 = f_0$ and $f\ii n(x) = f_n(x) - f_{n-1}(x)$ for
$n\geq 1$.

Given a real number $\gamma$, $0 < \gamma < 1$, we define for
$f\in L^2(\X,\nu)$ a norm $\norm f_\gamma$ by
\begin{equation}\label{gammafunspace}
  \norm f_\gamma = 
  \sum_{n=0}^\infty \gamma^{-n} \norm {f\ii n}.
\end{equation}
The space of functions $f: X \to \RR$ such that the $\gamma$-norm
$\norm f_\gamma$ is finite is denoted $L^2_\gamma$.  We observe that
$L^2_\gamma$ is a Banach space which is dense in $L^2(\X,\nu)$. One
can perhaps think of it as a type of Besov space.

Note also that if $f: X \to \RR$ is $\al$-H\"older continuous in the
sense that
$$ \var_n f = \sup_{\substack{[x]_n=[y]_n}} |f(x)-f(y)| = \Ordo{2^{-\al n}} $$
then it belongs to $L^2_\gamma$ for $\gamma > 2^{-\al}$.

Our main result involves proving a spectral gap for the transfer
operator $L$ if we restrict $L$ to the spaces $L^2_\gamma$.  Let
$c_k$, $0\leq c_k < 1$ be as in the irreducibility condition
\qr{irred}.

\begin{theorem}\label{thm1}
  Assume that the irreducibility condition \qr{irred} holds and that
  $\gamma$ is as above.  Define
  $$ \theta_2 := \inf_k (1-c_k)^{1/k} < 1. $$
  Then, for $\gamma > \theta_2$, the transfer operator $L$ restricts
  to a continuous operator $L: L^2_\gamma \to L^2_\gamma$ having a
  spectral gap.
\end{theorem}
\begin{remark}
  The bound $\theta_2=(1-c_k)^{1/k}$ is not meant to be optimal; it is
  based on an argument where we use a pointwise estimate of a
  convergence rate $\Q_n G$ of projections.
\end{remark}

As a simple consequence of our results we have the following result,
which expresses how quickly the Kusuoka measure can be approximated.
\begin{corollary}\label{cor1}
  There exists $0 < \rho < 1$ such that for any $f\in L_\gamma^2$
  $$
  \left\| L^m f - \int f d\mu \right\|_{L_\gamma^2} \leq C \rho^m
  \left\|f\right\|_{L_\gamma^2}
  $$
  where $C$ is a uniform constant.
\end{corollary}
The rate of convergence, $\rho$, depends on both $\theta_2$ and a
contraction constant $\theta_1$ in \eqref{norminf}, or, equivalently,
\eqref{gol1}. From Theorem~\ref{contr1} it follows that we may choose
$\rho>\max\{\theta_1,\gamma\}$, where $\gamma>\theta_2$, as above.

As a consequence, we have exponential decay of correlations:
\begin{corollary}\label{cor3}
  If $f\in L^2(\X,\nu)$ and $g \in L_\gamma^2$ (e.g., an
  $\alpha$-H\"older continuous function, if $\gamma = 2^{-\alpha}$),
  then for some $0 < \rho < 1$ and some uniform constant $C$, we have
$$\left|\int f \, \left(g\circ T^n\right) \; d\nu- \int f\; d\nu \int g\; d\nu\right|\leq C\rho^n.$$
\end{corollary}

\subsection{Specialisation to Sierpi\'nski gaskets}
We now specialize to some explicit estimates of rates of convergence
for the cases that correspond to the Sierpi\'nski gasket $SG$ and the
family $SG_n$, $n=2, 3, \ldots$ ($SG=SG_2$), which are realized in
${\mathbb R}^2$ and constructed by $n(n+1)/2$ contraction mappings $F_j(x)=x/n  +b_{j,n}$
for suitable choices of $b_{j,n}$, so that $SG_n$ is the unique nonempty compact set that satisfies
$$SG_n=\bigcup_{j=1}^{\frac{1}{2}n(n+1)} F_j(SG_n).$$

By a direct computation of $\theta_1$, we obtain the
following result for $SG$. In Example 3 have included
the corresponding rates for $SG_n$, $n=3,4,5$. These explicit approximation results
depends on the fact that for $SG_n$ we have symmetric restriction maps
$A_s$.
\begin{theorem}\label{cor4}
  For any function $f$ which is $\CF_k$-measurable (i.e., measurable
  with respect to the finite algebra generated by cylinders of length
  $k$), we have
  \begin{equation}\label{SG}
    \left\|L^{n+k} f(x)-\int f\; d\nu\right\|_{\infty}\leq 2\left(\frac{4}{5}\right)^n \left\|f\right\|_1,
  \end{equation}
  where $\| \cdot \|_{\infty}$ denotes the essential supremum norm.
\end{theorem}
\begin{corollary}\label{cor5}
  If $\nu$ is the Kusuoka measure on three symbols related to the
  $SG$, we have for $A\in \CF_k$ and $B\in \CF$
  that
$$\left|\nu(\T^{-(n+k)}A \cap B)-\nu(A)\nu(B)\right|\leq 2\left(\frac{4}{5}\right)^n.$$
\end{corollary}
\begin{remark}
  In this simplified case, the rate of convergence can be expressed in
  terms of $\theta_1=\frac{4}{5}$ only. In Theorem \ref{thm1} we also
  need to consider the constant $\theta_2$ from Lemma \ref{thmD} below
  in order to obtain the uniform rate of convergence expressed, e.g.,
  in Corollary \ref{cor1}. Notice that in Theorem \ref{cor4} we use
  members of $\CF_k$ as test functions and we need to start the
  convergence at this level $k$, whereas in Corollary \ref{cor1} we
  may use any $f\in L_\gamma^2$ and we do not relate the number of
  iterates to the (lack of) regularity of $f$. Nevertheless, Theorem
  \ref{cor4} may give some insight about the rate of convergence from
  a practical point of view. We have given an argument for this
  special case only for $SG$, just in order to simplify matters, but a
  similar argument for this type of result for convergence in the
  $\|\cdot \|_\infty$-norm may be devised in the general Kusuoka
  measure case. The constant $2$ in front of the $(\frac{4}{5})^n$
  can be interpreted as the dimension of the space of harmonic
  functions modulo constants, i.e., the number of boundary points
  minus one. That we do not have a ``general'' uniform, but unknown, 
  constant $C$ in front of the $(\frac{4}{5})^n$ is due to a strict contraction 
  result, Lemma \ref{lem:strict}, which we have obtained for symmetric
  restriction maps, and which is thus valid for all $SG_n$.
\end{remark}

\begin{remark}
  If we have the probability weights
  $p_j(x)= \frac{1}{15}+\frac{12}{15}\frac{d\nu_j}{d\nu}$, $j=0,1,2$,
  for the iterated function system $\{F_j\}_{j=0}^2$ that defines the
  Sierpi\'nski gasket (as in Bell, Ho and Strichartz \cite{str3}), where
  $\nu_j$ are the energy measures so that the Kusuoka measure
  $\nu=\sum_{j=0}^2 \nu_j$, then a standard conjugation between
  symbolic space and the fractal $SG$ gives the same rate of
  convergence (namely $(4/5)^n$ for $SG$) with respect to the
  essential supremum norm for an associated transfer operator defined
  on H\"older continuous functions $f$ on $SG$ as
  $Lf(x)=\sum_{j=0}^2 p_j(x)f(F_j(x))$. Notice that we can view the
  natural extension of the full left shift on the symbol space as an
  iterated function system, where the probability $p_j(x)$ of choosing
  the symbol $j$ to go from the state $x=(x_0, x_1,\ldots)$ to
  $(j, x_0, x_1,\ldots)$ is given by $g(jx)$.
\end{remark}

\begin{example}\label{SG}
  We have explicitly computed the rate of convergence in the case of $SG_n$. For
  $SG_3$, the level 3 Sierpi\'nski gasket, is generated by the
  iterated function system $F_j(x)=\frac{1}{3}x+\frac{2}{3}v_j$,
  $j=1, 2, 3, 4, 5, 6$, where $v_1, v_2, v_3$ are the vertices of an
  equilateral triangle and where $v_4=\frac{v_2+v_3}{2}$,
  $v_5=\frac{v_1+v_3}{2}$, $v_6=\frac{v_1+v_2}{2}$. We approximate
  $SG_3$ with a graph sequence, and we use the same two initial
  orthonormal harmonic functions as in the case of $SG$ (the three
  boundary points are the same). We obtain two families of matrices
  (the restriction maps) $A_s$ with three in each family being
  rotations by $120^{\circ}$ of each other. That is, we have one
  family of three matrices that restricts values to the three
  triangles with one vertex at the original vertex points
  $v_1, v_2, v_3$ and another family of three matrices that restricts
  values to the three other triangles. There are similar and obvious
  ways to describe the other fractals in the family $SG_n$.

  In these cases we get $\theta_1=\frac{5}{7}$ for $SG_3$,
  $\theta_1=\frac{2822}{4223}$ for $SG_4$ and
  $\theta_1=\frac{209527}{327611}$ for $SG_5$.
\end{example}

\section{The transfer operator on the space $\VV$}

Instead of working with the $L^2$-space of regular functions on $\X$,
the idea is to work with a Hilbert space $\VV$ consisting of
``operator valued process limits'', where the shift operator $\T$ is
defined. The action of the corresponding transfer operator has a
simple explicit description.  The Hilbert space $L^2(\X,\nu)$ has a
representation as a subspace $\LL$ of $\VV$. The space $\LL$ itself is
not invariant under the action of $\L$, but the transfer operator $L$
on $L^2(\X,\nu)$ can be recovered as a dilation such that
$L^k=\Q \circ\L^k$, where $\Q$ is the orthogonal projection onto
$\LL$. We show in subsection \ref{cont-l} that $\Q$ is continuous
with respect to the $\gamma$-norm, for suitable $\gamma$, and, hence,
that the spectral properties of $\L$ on $\VV$ carry over to results on
the action of $L$ on the spaces $L^2(\X,\nu)$.

\subsection{Construction of a graded Hilbert space of process limits}

We will use a general construction of a certain graded Hilbert space
of process limits on $\X$ under a given system of ``restriction
operators'' $\psi_s$, $s\in S$. The Hilbert space is modeled by the
martingale representation of functions in $L^2(\X,\mu)$, but with the
difference that they do not necessarily converge to functions. The
construction can be generalised to non-self similar systems using a
systematic approach based on direct and inverse limits.

\subsubsection{$E$-valued processes and finite degree process limits}

Let $E$ denote a finite-dimensional linear space. An \emph{$E$-valued
  process} is a function $f: S^*\to E$, where $S^*$ is the set of
initial cylinders. A process $f(\al)$, $\al\in S^*$, can be identified
with the sequence $f_n(x)$ of $\CF_n$-measurable functions
$f_n(x)=f([x]_n)$.

Let $\EE_0$ denote the direct sum $\bigoplus_{\al} E_\al$, where for
all initial cylinders $\alpha\in S^*$, $E_\al$ is a copy of $E$.  We
interpret $\EE_0$ as the space of all processes $f(\al)$ such that
there is a smallest integer $\deg(f)\geq 0$ where $f(\al)=0$ for all
cylinders $\al$ of length $\ell(\al) > \deg(f)$. We refer to $\deg(f)$
as the \emph{degree} of $f\in \EE_0$.

We will need a construction which, more formally (and more generally),
involves taking direct and inverse limits.  Given a set
$\psi=\{\psi_s: s\in S\}$ (a \emph{self similar} system of restriction
maps) of linear maps $\psi_s: E\to E$, let $\varPsi:\EE_0\to\EE_0$ be
the map $f\mapsto \varPsi f$ given by
$$ (\varPsi f) (\al s) = \psi_s f(\al), \quad 
(\varPsi f)(\emptyset)=f(\emptyset). $$
We assume that $\varPsi:\EE_0\to\EE_0$ is an injective map.

Let $\EE_*$ be the space $\EE_0=\bigoplus_\al E_\al$ modulo the
subspace $\opn{Ker} (I-\varPsi)$. The space $\EE_*$ is the space of
limit orbits for the map $\varPsi$. Elements $f$ in $\EE_*$ can be
represented by $E$-valued processes which are ``eventually constant''
in the following sense: There is a smallest number $\deg(f)\geq0$ where
$f(\al s)=\psi_sf(\al)$ for $\ell(\al)>\deg(f)$. Two processes,
$f(\al)$ and $g(\al)$ of degree at most $n$ are identified if
$f(\al)=g(\al)$ for all $\al\in S^n$. The finite dimensional space
$\EE_n$ of elements $f\in\EE_*$ of degree $\deg(f)\leq n$, consists of
those elements $f\in\EE_*$ that can be written on the form
$f = \tilde f+\Ker(I-\varPsi)$ where $\tilde f\in\EE_0$ has degree
less than or equal to $n$.

\subsubsection{Invariant bilinear forms}

A bilinear form $\mscr F$ on $\EE_*$ is \emph{local} if it has the
form
\begin{equation}
  \mscr F(f,g) = \lim_{n\to\infty} \sum_{\al\in S^n} \mscr F(\al)(f(\al),g(\al)), 
  \quad\forall\, f,g\in \EE_*.\label{local}
\end{equation}
where, for each $\al$, $\mscr F(\al)$ is a given form on $E$. The
local form $\mscr F$ on $\EE_*$ is well defined if, for
all $\al\in S^*$, we have the invariance condition
\begin{equation}
  \mscr F(\al)(u,v) = \sum_s \mscr F(\al s)(\psi_s u, \psi_s v), 
  \quad\forall\, u,v\in E.\label{invariant0}
\end{equation}
In this case the limit in \qr{local} is the limit of an eventually
constant sequence.

In particular, a given fixed form $\mscr E$ on $E$ gives a
\emph{constant} invariant form on $\EE_*$ if and only if
\begin{equation}
  \E(u,v) = \sum_s \E(\psi_s u, \psi_s v), \quad\forall\, u,v\in E.\label{invariant}
\end{equation}

\subsubsection{The Hilbert space of $E$-valued processes and the
  orthogonal decomposition}
Given a restriction system $\psi=\{\psi_s\}$ and an local invariant
positive definite form $\mscr F$ satisfying \qr{invariant0}, 
we obtain a non-degenerate inner product $\scp fg_\EE$ on
$\EE_*$ as the limit in \qr{local}. We let $\EE = \EE(\psi,\mscr F)$
be the Hilbert space obtained as the completion of $\EE_*$ with
respect to the norm 
$$\| f \|_{\EE} = (\scp ff_\EE)^{1/2}. $$

The spaces $\EE_n$ of processes of degree less than $n$ are closed
subspaces of $\EE$. We can define $\EE\ii n:=\EE_n\ominus \EE_{n-1}$
as the orthogonal complement of $\EE_{n-1}$ inside $\EE_n$. This gives
us an orthogonal decomposition of $\EE$, so that any $f\in\EE$ has a
unique expression $f = \sum_{n=0}^\infty f\ii n$, where
$f\ii n\in \EE\ii n$ and
$\| f \|^2_\EE = \sum_{n=0}^\infty \| f\ii n\|^2_\EE$.  This
orthogonal decomposition lets us express a process limit $f$ by a
\emph{unique} representative process
$$ f(\alpha)= f\ii0(\al)+\dots + f\ii
n(\al),\quad\text{for } \al\in S^n. $$

For a number $\gamma\in(0,1)$, we define the Banach space $\EE_\gamma$
with the norm $\norm x_\gamma$
$$ \norm{f}_{\EE_\gamma} = \sum_{i=0}^\infty \gamma^{-i} \norm{f\ii i}_\EE. $$

\subsubsection{The martingale representation of
  $L^2(\X,\mu)$}\label{sec:mgalconstr}

Note that for any probability measure $\mu$ on $\X$, we can construct
the martingale representation of the space $L^2(\X,\mu)$ as a graded
Hilbert space according to the scheme above as follows: Let $E$ be the
space $\RR$ of real numbers and let the restriction system
$\{\psi_s\}$ be given by $\psi_s(x)=x$. This gives the usual
restriction of functions and the corresponding limits in $\EE_*$ are
locally constant functions: The spaces $\EE_n$, $n\geq 0$, will
correspond to the spaces of $\CF_n$-measurable functions.  We use a
non-constant local invariant form $\mscr F(\al)$ given by
$\mscr F(\al)(x,y) = \mu(\al) xy$ and the invariance condition
\qr{invariant0} holds since
$$ \sum_s \mu(\al s) xy = \mu(\al) xy. $$
Since, for $f,g\in\EE_n$, the limit $\scp fg_\EE$ in \qr{local}, gives
$$ \scp fg_\EE = \sum_\alpha f(\al) g(\al) \mu(\al) = \int f(x)g(x)\d\mu(x).$$
The closure $\EE$ of $\EE_*$ will hence give an isometric copy of the
space $L^2(\X,\mu)$.

\subsection{The operator valued Hilbert space $\VV$}

We show (by copying some arguments given by Kusuoka in
\cite{kusuoka2}) that, starting from a system $\{\a_s\}$ of
restriction maps $\a_s: H\to H$ on a finite dimensional space $E$, we
can define a Hilbert space $\VV = \VV(\{\a_s\})$ of process limits
taking values in the space $\B(E)$ of linear operators on $E$. The
inner product $\Scp \cdot\cdot$ on $\VV$ is induced by a constant
bi-invariant (see \qr{lrinvariance}) positive definite bilinear form
$\scpe\cdot\cdot$ on $\B$. The space $\VV$ also allow us to define the
shift operator $\T f$, $f\in\VV$, as an isometric injective map.

Kusuoka's paper \cite{kusuoka2} starts with a self-similar system
$\{\a_s\}$ of injective maps on a finite dimensional space $H$ which
is irreducible in the sense \qr{irredorig}.  It is proved that, modulo
a re-scaling (i.e.\ we replace $\a_s$ with $\lambda \a_s$ for some
$\lambda>0$), there exists a unique invariant positive definite form
$\E$ on $H$. In addition, there is a corresponding \emph{dual
  invariant} positive definite form $\mscr R$ defined on $H^*$, such
that
\begin{equation}
  \mscr R(z,w) = \sum_s \mscr R(\a_s^* z, \a_s^* w), 
  \quad\forall\, z,w\in H^*.\label{dualinvariant}
\end{equation}
From the pair $\E$ and $\mscr R$, we define the inner product
$\scpe AB$ on the space of operators $\B(E) \cong H\otimes H^*$ by
setting
\begin{equation}\label{opinv}
  \scpe {u\otimes v^*}{f\otimes g^*} = \E(u,f)\mscr R(v^*,g^*),
  \quad u,f\in H,\ v^*,g^*\in H^*,
\end{equation}
for rank one operators and then extend it by bilinearity. If we have a
Hilbert space structure on $H$ (and $H^*$) so that $\E$ and $\R$ are
represented as symmetric operators in $\B$ then the inner product is
given by $\scpe AB = \Tr(B^*\E A \R)$. In particular, if we assume that the
inner product on $H$ is the form $\R$ --- so that $\R$ is represented by the
identity operator --- then we see that \qr{invariant} and
\qr{dualinvariant} take the forms \qr{fundeq1} and \qr{fundeq2},
which were our starting points.

The system $\{\a_s\}$ acts on $F\in\B(E)$ both from the left,
$\a_s F = \a_s \circ F$, and from the right, $F\a_s = F\circ\a_s$. The
form $\scpe FG$ is then both left and right invariant, i.e.\
\begin{equation}\label{lrinvariance}
  \scpe FG = \sum_s \scpe{\a_s F}{\a_s G} = \sum_s \scpe{F\a_s}{G\a_s}, 
\end{equation}
on account of $\E$ satisfying \qr{invariant} and $\mscr R$ satisfying
\qr{dualinvariant}.  In particular, we can consistently define the
corresponding Kusuoka measure on $(\X,\CF)$ by taking
$\nu(\al)=\scpe{\A(\al)}{\A(\al)}$, where $\A(\al)$ denotes the
composition $\a_{\al_1}\circ \dots\circ \a_{\al_n}$.

The left invariance of $\scpe\cdot\cdot$ in \qr{lrinvariance}, states
that the form $\scpe\cdot\cdot$ on the finite dimensional space $\B$
is invariant with respect to the restriction system
$\psi_s(F)=\a_s F$. We obtain, by the general construction above, a
graded Hilbert space $\VV$ with a scalar product
\begin{equation}\label{Scpdef}
  \Scp FG = \lim_{n\to\infty} \sum_{\al\in S^{n}} \scpe {F(\al)}{G(\al)}. 
\end{equation}
The grading means that $\VV$ has the orthogonal decomposition
$\oplus_{n\geq 0}\VV\ii n$ so that
$$\Scp FG = \sum_{n\geq0} \Scp{F\ii n}{G\ii n}. $$ 

An element $G(\al s)\in \VV_n$, $\al s\in S^n$, belongs to $\VV\ii n$,
$n\geq 1$, if and only if the relation
\begin{equation}
  \sum_s \a^*_s \,\E\, G(\al s)  = \mb 0 \label{orth1}
\end{equation}
holds for all $\al \in S^{n-1}$. This follows since if
$F(\alpha)\in\VV_{n-1}$ then
$$ 
\Scp FG = \sum_\al \sum_s \scpe{G(\al s)}{\a_s F(\al)} = \sum_\al
\Tr\left( \left(\sum_s \a_s^* \E G(\al s)\right){F(\al)}\right).
$$ 
This can be zero for arbitrary $F$ only if $G$ satisfies \qr{orth1}.

A process $F\in\VV$ belongs to $\VV\ii0$ if and only if
$F(\al) = \Al F_0$ for some constant operator $F_0=F(\emptyset)\in\B$.
The process $\A(\al)$ denotes the ``identity process'' $\A\in\VV\ii0$
with $\A(\emptyset)=I$.

\subsection{The shift operator and the transfer operator on $\VV$ and a spectral gap}
Because of the right invariance in \qr{lrinvariance}, we can
furthermore consistently define the left shift operator $\T$ as the
injective map $\T:\VV \to \VV$ given by
\begin{equation}\label{defTonBB}
  \T F (s\alpha) = F(\alpha)\a_s. 
\end{equation}
It is an isometric embedding in the sense that
$\Scp FG = \Scp{\T F}{\T G}$.  Note also that if $F\in \VV_*$ has
finite degree, then $\deg (\T F)=\deg (F) + 1$ and that the space
$\VV\ii k$, $k\geq 0$, is embedded by $\T$ into the space
$\VV\ii{k+1}$, since $\T$ manifestly preserves the orthogonality
condition \qr{orth1}.

The \emph{transfer operator} $\L:\VV\to\VV$ is defined as the dual of
$\T$, i.e.\ $\L:\VV\to\VV$ satisfies $\Scp{\L F}{G}=\Scp{F}{\T G}$ for
all $F,G\in\VV$.  For $F,G\in \VV_*$, we have
\begin{align*}
  \Scp F{\T G} 
  &= \sum_\al\sum_s \Tr\left(\a_s^* G(\al)^*\E F(s\al)\right)\\
  &= \sum_\al \Tr\left(G(\al)^*\, \E\, (\sum_s F(s\al)\a_s^*)\right)
\end{align*}
provided $\ell(\al)$ is large enough.  It follows that the operator
$\L:\VV \to \VV$ is explicitly given by the expression
\begin{equation}\label{explicitL}
  (\L F)(\al) = \sum_{s\in S} F(s\al) \a_s^*,
\end{equation}
which is a simple multiplicative operator, analogous to the transfer
operator for a Bernoulli measure.

Since it is dual to the isometric embedding $\T$, it must be that $\L$
is a contraction, i.e.\ $\|\L F\|\leq \|F\|$. From \qr{fundeq2}, it is
also clear that the process $\A(\al)$ is an eigenvector to $\L$
corresponding to eigenvalue $\lambda=1$, since
$\sum_s \A(\al)\a_s\a_s^*=\A(\al)$.  The operator $\L$ acts as a
reverse shift operator on the spaces $\T^k(\VV)$, $k>0$. Its essential
spectral radius is therefore $1$.

Note also that $\L$ preserves the space of constants $\VV\ii 0$: From
\qr{explicitL} it follows that if $G(\al)=\Al G_0$, $\al\in S^*$, then
\begin{equation}\label{LonV0}
\L G(\al) = \Al \M(G_0), 
\end{equation}
where, $\M:\B\to\B$ is the operator defined in \qr{Mdef}.

For $0<\gamma\leq1$, we define the Banach space $\VV_\gamma$ by the
norm
\begin{equation}
  \label{gammaopdef}
  \|F\|_\gamma = \sum_{n=0}^\infty \gamma^{-n}\| F\ii n \|.
\end{equation}
We want to show that the operator $\L_\gamma=\L\vert_{\VV_\gamma}$,
i.e.\ $\L$ restricted to $\VV_\gamma$, has a spectral gap.  Recall
that the spectrum $\sigma(\L)$ of an operator $\L:\VV\to\VV$ is the
set of complex numbers $z$ such that the operator $zI-\L$ is not
invertible. 

\begin{theorem}\label{contr1}
  Assume $0<\gamma<1$ and consider the operator $\L_\gamma$ of $\L$
  restricted to $\VV_\gamma$. Then $\L_\gamma$ has the eigenvalue
  $\lambda=1$ corresponding to the unique eigenvector $\A$. There is
  also a constant $\theta_1<1$, such that
  $\sigma(\L_\gamma)\setminus\{1\}$ is contained in the disc of radius
  $\rho= \max\{\gamma,\theta_1\}$.
\end{theorem}
Note that the essential spectral radius of $\L_\gamma$ is $\gamma$.

\subsection{Proof of Theorem~\ref{contr1}}
We prove the following two bounds. Firstly that
\begin{equation}\label{gol0}
  \| \L X \|_\gamma  = \gamma^{-k+1}\| (\L X)\ii {k-1} \| \leq \gamma \| X \|_\gamma, 
  \quad\forall\, X\in \VV\ii k,\ k\geq 1.
\end{equation}

Secondly, we show that for some $\theta_1<1$ and some $C\geq 1$, we have that
\begin{equation}\label{gol1}
  \| \L^k X \| \leq C\theta_1^k \, \| X \|, 
  \quad\forall\, X\in \VV\ii 0,\quad \Scp X\A = 0. 
\end{equation}

For $\rho=\max\{\theta_1,\gamma\}$, it follows from \qr{gol0} and
\qr{gol1} that for $Y=Y^{(0)}+Y^{(1)}+\cdots$, such that $\Scp Y\A=(Y^{(0)}|\A)=0$, we have 
\begin{align*}
  \| \L^k Y \|_\gamma &\leq \sum_{j=0}^\infty \|\L^k Y\ii j\|_\gamma \\
                    &\leq 
                      \sum_{j=0}^k \gamma^{j}\cdot C\theta_1^{k-j} \| Y\ii j \|_\gamma +
                      \sum_{j=k+1}^\infty \gamma^{k} \cdot \| Y\ii j \|_\gamma.
\end{align*}   
Here we use that $(\L^j Y^{(j)}  | \A)=0$ which follows from the observation that
$$(\L X|\A)=(X|T\A)=(X|\A).$$ Hence,                
\begin{equation*}                    
                     \| \L^k Y \|_\gamma \leq C\rho^k \cdot \sum_{j=0}^\infty \| Y\ii j\|_\gamma
                    = C \rho^k \cdot \norm{Y}_\gamma.
\end{equation*}
This shows that the spectral radius of $\L$ restricted to the space of the elements $Y$ in $\VV_\gamma$ such that $(Y | \A)=0$ is less than $\rho$.

\subsubsection{Proof of the bound \qr{gol0}}
The bound \qr{gol0} is a consequence of $\L$ being a contraction, that is
$$\| \L X \|\leq \| X \|. $$
Since $\T$ is an isometric embedding any ON-basis
$\{E_i\}$ of $\VV$ is transported to
an ON-basis $\{\T E_i\}$ of $\T(\VV)$ and, by Parseval's identity, the
squared norm satisfies
\begin{equation}\label{parsid}
  \| \L X \|^2 = \sum_i \Scp{\L X}{E_i}^2 = \sum_i \Scp{X}{\T E_i}^2 \leq \| X \|^2.
\end{equation}

If $X\in \VV\ii k$, $k\geq 0$, then we find that
$\L X \in \VV\ii{k-1}$, since the dual operator, $\T$, restricts to an
isometric embedding of $\VV\ii {k-1}$ into $\VV\ii k$. Taking an
ON-basis $E_l\ii j$ of $\VV\ii j$ and using \qr{parsid} shows that
$$ 
\|\L X\| = \|(\L X)\ii {k-1}\|= \|\L \circ \opn{proj}_W X\| \leq \| X
\|,
$$ 
where $\opn{proj}_W$ is the orthogonal projection onto the subspace
$W=\T(\VV\ii{k-1})$. \qed

\subsubsection{Proof of the bound \qr{gol1}}
For convenience, we extend the setting to complex matrices in order 
to include the case of anti-symmetric matrices.
 Let $\mscr D$ be the space of Hermitian (self-adjoint) operators
$B$ such that $\scpe BI=\Tr(\E B)=0$. We extend the operator $\M$ to complex matrices 
and $\mscr D$ is then an $\M$-invariant subspace, since
$$\Tr(\E \M(B)I) = \Tr(\M^*(\E) B) = \Tr(\E B) $$ 
on account of \qr{fundeq1}.

\begin{lemma}\label{c}
Assume \eqref{irredorig} holds. There are constants $\theta_1<1$ such that for any 
$B\in\mscr D$ we have
\begin{equation}\label{norminf}
  \norme{\M^k(B)} \leq C_0 \theta_1^k \norme B,
\end{equation}
for some $C_0>0$. 
\end{lemma}
In order to show \qr{gol1} it is enough, by \qr{LonV0}, to consider the operator $\M$ in \qr{Mdef}
acting on $\B$. Since any element in $\B$ uniquely can
be represented as an orthogonal sum of an symmetric and anti-symmetric operator, it
suffices to analyse the action of $\M$ restricted to $\mscr D$, since the map $B\mapsto i\cdot B$ is an
isomorphism between the space $\mscr D$ and the $\M$-invariant space
of anti-Hermitian matrices. 

\subsubsection{Proof of Lemma~\ref{c}}
We can take $\theta_1$ as the maximum eigenvalue of the operator $\M$
restricted to $\mscr D$.

For a Hermitian operator $B\in\mscr D$, let
$\sigma(B)=\{\lambda_i\}\subset \RR$ denote its spectrum.  Let
$\sigma(\M(B))=\{\gamma_j\}\subset \RR$ denote the spectrum of
$\M(B)\in\mscr D$.  We have the spectral decompositions
\begin{equation}\label{specdec}
  B = \sum_{i} \lambda_i P_i \quad\text{and}\quad
  \M(B)=\sum_i \lambda_i \M(P_i)=\sum_{j} \gamma_j Q_j,
\end{equation}
where $P_i$ and $Q_j$ refers to systems of orthogonal projections such
that $\sum_i P_i=\sum_j Q_j=I$.  Let, for the moment,
$\scp AB:=\Tr(AB^*)$ denote the Hilbert-Schmidt scalar product on
$\B=\B(H)$. Then $\{P_i\}$ and $\{Q_j\}$ are orthogonal sets under
$\scp\cdot\cdot$. (We have $P_iP_{i'}=0$ if $i\not=i'$.)  Taking the
orthogonal projection in the direction of $Q_j$ of both sides in
\qr{specdec} gives, for each $j$, the equation
\begin{equation}\label{eqlg}
  \left(\sum_{i} c_{ij} \lambda_i\right) Q_j   = \gamma_j Q_j
\end{equation}
where
\begin{equation}\label{cijsum}
  c_{ij} =
  \frac{\scp{\M(P_i)}{Q_j}}{\scp{Q_j}{Q_j}} \quad\text{and}\quad
  \sum_{i} c_{ij} = \frac{\scp{I}{Q_j}}{\scp{Q_j}{Q_j}}=1,
\end{equation}
since $\sum_i \M(P_i)=\M(I)=I$ by \qr{fundeq2}.  The
coefficients $c_{ij}$ cannot be negative, since
\begin{equation}
  \scp{\M(P_i)}{Q_j}=\Tr(\M(P_i)Q_j)=\Tr(Q_j\,\M(P_i)\, Q_j)\geq0.\label{positive}
\end{equation}
This is a consequence of the positivity of $\M$, i.e.\ that $\M$
preserves the cone of positive semidefinite matrices so that
$\M(P_i)$, and hence $Q_j\M(P_i)\,Q_j$, both are positive
semi-definite.

Thus \qr{eqlg} expresses each eigenvalue $\gamma_j$ of $\M(B)$ as a
convex combination of the eigenvalues $\{\lambda_j\}$. In particular,
if we order them so that $\gamma_1 > \gamma_2 > \dots $ and
$\lambda_1>\lambda_2>\dots$ then
$$\gamma_1 =  \sum_i c_{i1}\lambda_i \leq \lambda_1. $$

Since $\M$ is an operator on finite dimensional space, there is some
eigenvector $B\in\mscr D$ corresponding to $\theta\in\sigma(\M)$ of
maximum modulus.  If we assume that $B$ is an eigenvector of $\M$ then
$\gamma_j=\theta \lambda_j$ and we can assume that $Q_j=P_j$.

If we assume that $|\theta|=1$, then it must hold that $c_{11}=1$ and
$c_{i1}=0$ for $i=2,\dots$. But that is equivalent to the equalities
\begin{equation}\label{scpPQ}
  \scp{\M(P_1)}{P_1}= \scp{P_1}{P_1}\quad\text{and}\quad
  \scp{\M(P_1)}{I-P_1}= 0. 
\end{equation}
A consequence of positivity is that $P_i \M(I-P_i) P_i \geq0$ and
thus, since $\M(I)=I$, that $0\leq P_i\M(P_i)P_i\leq P_i$. Hence, it
follows from \qr{scpPQ} that, in fact, $\M(P_1)=P_1$.

Moreover, since each term in the sum
$$\M(P_1) = \sum_s \a_s\, P_1\a_s^* = P_1$$ is
positive definite, it follows that $\a_s^* (W) \subset W$ for all $s$,
where $W$ is the range of $P_1$. This contradicts the irreducibility
condition \qr{irredorig} unless $P_1=I$. However, that would imply
that $B$ is a scalar multiple of $I$, which contradicts the condition $\scpe BI=0$.  \qed

\subsection{Strict contraction in Schatten norms}\label{sec:strictsym}

If the operators $\{\a_s\}$ are \emph{symmetric}, then the scalar
product $\scpe \cdot\cdot$ is proportional to $\scp\cdot\cdot$. In
other words $\E=\frac 1d I$, where $d$ is the dimension of $H$ and, in
particular, we have the identity 
\begin{equation}\label{eq:MHS}
  \Tr(\M(B)) = \scp {\M(B)}{I}=\scp BI = \Tr(B),
\end{equation}
since this holds for $\scpe\cdot\cdot$. 

In this case, we obtain strict contractivity of $\M$ for all
Schatten norms $\norm\cdot_p$ on $\mscr D$, which we define for $p\geq
1$ as 
$$\norm A_p = \Tr( |A|^p)^{1/p} $$
where $|A|$ denotes the positive part of the operator $A$, i.e.\ the
unique positive definite operator $|A|$ such that $A=|A|R$ for some
orthonormal operator $R$.  For any Hermitian matrix $A$ we can write
$|A|=\sum_i |\lambda_i| P_i$, where $A=\sum_i \lambda_i P_i$ is the
spectral decomposition of $A$. The norm
$\norm A_p$ can then be expressed as 
$$\norm A_p=\left(\sum_i |\lambda_i|^p \scp {P_i}{P_i}\right)^{1/p}.$$
 
\begin{lemma}\label{lem:strict}
  Assume that the restriction maps $\a_s$ are all symmetric.  For all
  $p\geq 1$, there is a constant $\theta_{1,p}$, $0<\theta_{1,p}<1$,
  such that
\begin{equation}\label{strictcontr}
\norm{\M(B)}_p \leq \theta_{1,p} \norm{B}_p,
\end{equation}
for all $B\in\mscr D$. 
\end{lemma}
From the explicit action of $\M$ given in \qr{SGM} it is clear that
$\theta_{1,p} = 4/5$, for all $p$, if we restrict to the particular
case of the Sierpi\'nski gasket.

\begin{proof}[Proof of Lemma~\ref{lem:strict}]
  We use the notation from the argument showing Lemma~\ref{c}. 
  From the fact that $\gamma_j$ can be expressed as a convex
  combination of the $\lambda_i$s, we obtain
  \begin{align*}
    \norm{\M(B)}_p^p &= \sum_j \left|\sum_i c_{ij}\lambda_i\right|^p
                       \scp{Q_j}{Q_j} 
                    \leq \sum_j\sum_i c_{ij} |\lambda_i|^p\scp{Q_j}{Q_j}
  \end{align*}
 using Jensen's inequality.

 The definition $c_{ij}=\scp{\M(P_i)}{Q_j}/\scp{Q_j}{Q_j}$ of $c_{ij}$
 shows that
  \begin{align*}
    \norm{\M(B)}_p^p &\leq \sum_i |\lambda_i|^p \left(\sum_j\scp{\M(P_i)}{Q_j}\right)
        = \sum_i |\lambda_i|^p \scp{\M(P_i)}{I}\\
        &= \sum_i |\lambda_i|^p \scp{P_i}{P_i} = \norm B_p^p
  \end{align*}
  on account of \qr{eq:MHS}. 

  The irreducibility condition implies that Jensen's inequality must
  be strict for all $B$. Compactness leads us to deduce that $\M$ is a
  strict contraction on $\mscr D$ in the norm $\norm\cdot_p$.
\end{proof}

\section{Proofs of the main results}
\subsection{Proof of Theorem 1}
Consider the construction of $L^2(\X,\mu)$ given in section
\ref{sec:mgalconstr} above. In the case when the measure considered is
the Kusuoka measure $\nu$ for the system $\{\a_s\}$, we can represent
a function $f(x)\in L^2(\X,\nu)$ with an operator valued process limit
in $\VV$ by
\begin{equation}
  \Phi(f)(\al) = f(\al) \Al,\label{Phidef}
\end{equation}
where $f(\al)$ is the martingale process corresponding to $f$.  This
becomes an \emph{isometry}, since we have
\begin{equation}\label{Phiisiso}
  \mscr F(\al)(f(\al),g(\al)) = \nu(\al) f(\al)g(\al) = 
  \scpe {f(\al)\Al}{g(\al)\Al}. 
\end{equation}
Hence, $\Phi$ gives an isometric representation of $\LL=L^2(\X,\nu)$
as a closed subspace in $\VV$.

Notice that $\Phi$ preserves the grading, so that
$\deg(\Phi(f))=\deg(f)$ for $f\in \LL_*$ and $\Phi(f)\in\VV_*$.  Hence
the representations of $f\in \LL_\gamma$ as $\Phi(f)\in\VV_\gamma$ are
isometric as well.  It is also clear that $\Phi$ commutes with the
shift operator, i.e.\
$$ 
\Phi(\T f)(s \al) = \A(s\al) f(\al) = \a_s\Phi(f)(\al) = \T
\left(\Phi(f)(s\al)\right).
$$
From now on we view the space $\LL$ as a closed subspace of $\VV$ and
drop the explicit use of $\Phi$.

For the qualitative results of this paper, it is inessential if we
replace \qr{irred} by the stronger condition that
\begin{equation}\label{irred2}
  c:=\inf_{F} \sum_s \scpe{\a_sF\a_s^*}{I}^2 >0,
\end{equation}
and then take $\theta_2=(1-c)^{1/2}$.  If \qr{irred} holds for a
certain $k>1$, we can then use an ``amalgamated'' symbolic space based
on the symbols $S'=S^k$ and the maps by $\a'_{s'}=\A(s')$, $s'\in S'$.
Moreover, using $\gamma'=\gamma^k$ will relate the $\gamma$-norms on
the original and the amalgamated system.

Let $\Q$ be the orthogonal projection of $\VV$ onto $\LL$.  The
transfer operator $L$ on $\LL$, is defined by the duality
\qr{dualityL}. By the (implied) isometry $\Phi$, we have
\begin{equation}\label{TLdual}
  \scp {L f}{g}_\LL = 
  \scp{f}{\T g}_\LL = 
  \Scp{f}{\T g} = \Scp{\L f}{g} 
  = \scp{\Q\circ\L f}{g}_\LL.
\end{equation}
The operator $L:\LL \to\LL$ can hence be expressed as the composition
$L= \Q \circ \L$.

From using $\L^k$ and $\T^k$ instead of $\L$ and $\T$ in \qr{TLdual},
we deduce furthermore that
$$L^k=\Q\circ\L^k, \quad \text{for $k\geq 1$.}$$ 
It follows that $R(s) = (L-s\ett)^{-1}$ is given by
$\Q \circ \mscr R(s)$, where $\mscr R(s) = (\L-s\ett)^{-1}$. If $\Q$
is a continuous map on $\VV_\gamma$ then
$\| R(s) \|_\gamma \leq C \| \mscr R(s) \|$.  The spectrum $\sigma(L)$
of $L$ is therefore contained in the spectrum $\sigma(\L)$ of $\L$.
Since Theorem~\ref{contr1} states that $\L\vert_{\VV_\gamma}$ has a
spectral gap for all $\gamma<1$, we deduce that Theorem~\ref{thm1}
holds for those $\gamma$ such that the operator $\Q$ is a well defined
and continuous on $\VV_\gamma$.

For a process $G(\alpha)\in\VV$, the projected process $\Q G$ has
explicitly the form
\begin{equation}\label{phidef0}
  (\Q G)(\al) = \nu(\al)^{-1} \scpe{G(\al)}{\Al}\Al.
\end{equation}
In other words, the value $\Q G(\al)$ is, locally for each $\alpha$,
the $\scpe\cdot\cdot$-orthogonal projection of the value $G(\al)\in\B$
onto the line in $\B$ spanned by $\Al$.  The explicit form
\qr{phidef0} follows since an orthornomal basis for
$\LL_n=\VV_n \cap \LL$ is given by
$$
\{\nu^{-1/2}(\beta) \A \, \ett_{\beta} : \beta\in S^n\, \},
$$
where $\ett_\beta$ denotes the real-valued process
$$
\ett_\beta(\al) =
\begin{cases} 1 & \text{if $\al = \beta\gamma$ for some $\gamma\in S^*$} \\
  0 &\text{otherwise}.
\end{cases}
$$
We have $\scpe{\Al}{\Al}=\nu(\al)$ so, from \qr{phidef0}, we deduce
that the squared norm $\| \Q_m F \|^2$ of a projection can be
expressed as
\begin{equation}\label{normQF}
  \| \Q_m F \|^2 = \sum_{\beta\in S^m} \scpe{F(\beta)}{\A(\beta)}^2.
\end{equation}
That the projection $\Q$ is continuous as a projection operator
between $\VV_\gamma$ and $\LL_\gamma$ is a nontrivial result since the
projection does not preserve the grading: The image of $\VV\ii k$
under $\Q$ spreads out on the spaces $\LL\ii j = \LL\cap \VV\ii j$,
for $j\geq k$. Hence, the norm $\| \Q F \|_{\VV_\gamma}$,
$F\in\VV_\gamma$, is not necessarily bounded in terms of
$\| F \|_{\VV_\gamma}$. We state and prove it as a lemma.

\subsection{The continuity of the projection}\label{cont-l}
\begin{lemma}[Continuity of $\Q$]\label{thmD}
  Let $\theta_2 = \sqrt{1-c}$ where $c$ is the constant in the
  irreducibility condition \qr{irred2}.  For any fixed $k$ and any
  $G\in\VV\ii k$
  $$ \| (\Q G)\ii j \| \leq \theta_2^{j-k} \| G \|. $$
\end{lemma}
In particular, we get, for any $G\in\VV_\gamma$, that
$$ \| \Q G \|_\gamma \leq \frac{1}{1-(\theta_2/\gamma)} \cdot \| G \|_\gamma, $$
provided $\gamma > \theta_2$.
Let $\Q_n$ denote the orthogonal projection onto the closed subspace
$\LL_n$. For $G$ in $\VV\ii k$, we have $\Q_n G = 0$ for $n<k$.  Let
$$Z^{[n]} = G-\Q_n G, \quad n\geq k $$
so that $\Q Z^{[n]} = \sum_{m>n} (\Q G)\ii m $ and
$(\Q G)\ii {n+1}=\Q_{n+1} Z^{[n]}$.
\newline

\noindent
{\bf Proof of Lemma~\ref{thmD}.}\newline
It is enough to show that, with $c$ as in the irreducibility condition
\qr{irred}, we have, for all $n>k$, that
\begin{equation}\label{g1}
  \| \Q_{n+1} Z^{[n]} \|^2 \geq c \| \Q Z^{[n]} \|^2.
\end{equation} 
By induction and orthogonality of $\Q_{n+1} Z^{[n]}$ and
$\Q Z^{[n+1]}$, we obtain
$$
\| \Q Z^{[{n+1}]} \|^2 = \| \Q Z^{[n]} \|^2 - \| \Q_{n+1} Z^{[n]} \|^2
\leq \theta^2_2 \| \Q Z^{[{n}]} \|^2
$$
and the sought after statement in Lemma~\ref{thmD} follows by induction,
since $\| \Q Z^{[k]} \| \leq \|G\|$.

Any process $F$ of degree $\deg(F)\leq n$, has the orthogonal
decompositions $$F=\sum_{\alpha\in S^n} F\ett_\alpha$$ of ``localised''
processes. Furthermore, the projections $\Q_m$, $m\geq 1$, respect
this localisation, i.e.\ $\Q_m F$ is the orthogonal sum
$\sum_\alpha \Q_m(F\ett_\alpha)$. It follows that it is enough to show
that
\begin{equation}\label{ZZgoal}
  \|\Q_{n+1} Z \|^2 \geq c \| \Q Z \|,
\end{equation}
for a part $Z=Z^{[n]}\ett_{\al}$ of $Z$, where $\al\in S^n$ is fixed.

Furthermore, for $\beta=\alpha\gamma\in S^m$, where $m\geq n$, we have
\begin{equation}\label{absc}
  \scpe{Z(\beta)}{\A(\beta)} = 
  \Tr(\E Z(\beta)\A(\alpha)^*\A(\gamma)^*) =
  \scpe{Z(\alpha\gamma)\A(\alpha)^*}{\A(\gamma)}.
\end{equation}
It follows that
$$\|\Q_m Z\| = \|\Q_{m-n} X\|, \quad\text{for $m\geq n$},$$
where $X(\gamma)=\A(\gamma) X_0\in \VV\ii0$ is a constant process with
$X_0=Z(\alpha)\A(\alpha)^*$.  Moreover, it follows from \qr{absc} that
$\Q_n X = \Q_n \tl X$, where $\tl X(\al) =\Al \frac12(X_0+X_0^*)$,
i.e\ the process based on the symmetric part of $X_0$: If $W\in \B$ is
anti-symmetric and $Y=\A W$ then $\Q Y=0$ since
$$ \scpe {\Al W}{\Al}=\Tr(\E\Al W \Al^*)=0. $$ 

It follows that to show \qr{ZZgoal} is equivalent to showing that
\begin{equation}\label{XXgoal}
  \| \Q_1 X \|^2 =  \sum_s \scpe{\a_s X_0 \a_s^*}{I}^2  \geq c \| \Q X \|^2, 
\end{equation}
where $X=\A X_0\in\VV\ii 0$ and $X_0$ is the symmetric part of
$Z^{[n]}(\al)\Al^*$. Since $Z^{[n]} = G - \Q_n G$, we moreover have
that
$$\scpe {X_0}I=\scpe{Z^{[n]}(\al)}{\Al}=0. $$ 
But, since $\Q$ is a projection, we have
$$ \| \Q X \|^2 \leq \| X \| = \norme{X_0}^2, $$
and thus \qr{XXgoal} is a direct consequence of the strong
irreducibility condition \qr{irred}. \qed

\subsection{Proof of Theorem \ref{cor4}}

\def\G{\varGamma}
\def\be{\beta}

In the case of the of the Sierpi\'nski gasket and its generalizations, the family $SG_n$, the restriction maps
$\a_s$, $s\in S$, are symmetric. It follows that the $\norme A$ is
$1/d$ times the Hilbert-Schmidt norm. Moreover, as is shown in
section \ref{sec:strictsym}, the symmetry of $\a_s$ also implies that $\M$ contracts
strictly in the Schatten-norms.

It follows directly from \qr{fundeq1} that
$$ H_n(x) =  \frac{A_n(x)^* \E A_n(x)}{\Tr(A_n(x)^*\E A_n(x))} $$ 
is a positive semi-definite and bounded matrix-valued $\nu$-martingale
process that, by the Martingale Convergence Theorem, converges $\nu$-almost everywhere to a limit $H(x)$ such that  
\begin{equation}\label{xg1} 
  \Tr(H(x))=1. 
\end{equation}
We can write 
$$ \nu(\al \mid [x]_n ) = 
\frac{\Tr(\E A_n(x)\Al \Al^* \A_n(x)}{\Tr(A_n(x)^*\E A_n(x))}
= \Tr(H_n(x) \Al\Al^*). $$ 
If we take the limit $n\to\infty$ we obtain
\begin{equation}\label{xxwx}
\nu(\al|x) = \Tr\left(H(x) \A(\al) \A(\al)^* \right) 
\end{equation}
$\nu$-almost everywhere. 

Assume now that $f$ is $\CF_k$-measurable function, $f(x)=f(\al)$ for
$\al\in S^{[0,k)}$.  By linearity of trace,
\begin{align*}
L^{m+k} f(x) &=\sum_{\substack{\al\in S^k\\ \be\in S^{[k,m+k)}}}
  \nu(\al\be|x)f(\al\be x) \\
  & =\sum_{\al}
\Tr\left(H(x) \left(\sum_\be \A(\be)
\A(\al) \A(\al)^* \A(\be)^*\right) \right) f(\al) \\
&=\sum_\al \opn{Tr} \left(H(x) \M^{m}\left(\A(\al) \A(\al)^*\right)\right)f(\al).
\end{align*}
Since $\M^m(I)= I$ and $\Tr(H(x))=1$, we obtain that 
\begin{align*}
L^{m+k} f(x) - \int f\d\nu 
&=\sum_\al \opn{Tr} \left(H(x) \M^{m}\left(\A(\al) \A(\al)^* - \nu(\al)I \right)\right)f(\al).
\end{align*}
For the proof of Theorem~\ref{cor4}, it thus remains to show that for
some constant $0<\theta_1<1$
\begin{equation}\label{goalfx}
\left| \Tr\left(H(x)\M^{m}(\A(\al) \A(\al)^* -\nu(\al)I)\right) \right|
\leq d \cdot \theta_1^{m} \cdot \nu(\al),
\end{equation}
uniformly for all $x$ 

Let $ B_\al$ denote the symmetric
zero-trace matrix $ B_\al=\A(\al) \A(\al)^* - \nu(\al)I$. By H\"older's
inequality for the Schatten norms we have
\begin{equation}\label{xxxn}
\left|\Tr\left(H(x)\M^m( B_\al)\right)\right| \leq
\norm{H(x)}_\infty \cdot \norm{\M^m( B_\al)}_1
\end{equation}
We have $\norm{H(x)}_\infty \leq \norm{H(x)}_1 =1$, by \qr{xg1}, and 
\qr{strictcontr} implies that
$$\Tr\left(|\M^m( B_\al)| \right)\leq \theta_{1}^m \Tr(| B_\al|) =
d\cdot\theta_{1}^m\, \Tr(\E| B_\al|), $$
where $\theta_1=\theta_{1,1}$. We then see that \qr{goalfx} follows from the estimate
$$ \Tr(\E | B_\al|) = \Tr(\E \cdot |\A(\al) \A(\al)^* -\nu(\al)I|) 
\leq \Tr(\E \A(\al) \A(\al)^*) =\nu(\al). $$
\qed

\end{document}